\documentclass{article}
\usepackage{amsmath, amsthm, epsfig, amssymb, multicol, tikz, array}
\usepackage{listings}
\usepackage{enumerate}

\newtheorem{lemma}{Lemma}
\newtheorem*{lemma*}{Lemma}
\newtheorem{theorem}{Theorem}
\newtheorem*{theorem*}{Theorem}
\newtheorem{corollary}{Corollary}
\newtheorem{definition}{Definition}
\newtheorem*{definition*}{Definition}

\newtheorem{prop}{Proposition}

\DeclareMathOperator{\supp}{supp}

\title{Strong Jumps and Lagrangians of Non-Uniform Hypergraphs}
\author{
Travis Johnston \thanks{University of South Carolina, Columbia, SC 29208, ({\tt j.travis.johnston@gmail.com}).}
\and Linyuan Lu \thanks{University of South Carolina, Columbia, SC
  29208, ({\tt lu@math.sc.edu}).
This author was supported in part by NSF grant  DMS 1300547 and ONR N00014-13-1-0717. }
}

\begin{document}

\maketitle

\begin{abstract}
  The hypergraph jump problem and the study of Lagrangians of uniform
  hypergraphs are two classical areas of study in the extremal graph
  theory.  In this paper, we refine the concept of jumps
  to strong jumps and consider the analogous problems over non-uniform
  hypergraphs.  Strong jumps have rich topological and algebraic
  structures.    The non-strong-jump values are
  precisely the densities of the hereditary properties, which include
  the Tur\'an densities of families of hypergraphs as special cases.
  Our method uses a generalized Lagrangian for non-uniform
  hypergraphs.  We also classify all strong jump values for
  $\{1,2\}$-hypergraphs.

\noindent{\em AMS classifications}: 05D05, 05C65, 05D40 \\
{\em Keywords}: Lagrangian, $R$-graph, non-uniform hypergraph, strong jump, non-jump, weak
jump, Tur\'an density, hereditary property
\end{abstract}

\section{Introduction}

The problem of determining jump values, or non-jump values, is intimately related to determining Tur\'an densities of hypergraphs.

\begin{definition}
	A real number $\alpha$ is a \textbf{jump} for a positive integer $r$ if there exists a $c>0$ such that for every $\epsilon>0$ and every $t\geq r$
	there exists an integer $n_{0}(\alpha, r, t, \epsilon)$ such that if $n\geq n_{0}$ and $G$ is an $r$-uniform hypergraph on $n$ vertices with edge
	density at least $\alpha+\epsilon$ then $G$ contains a subgraph $H$ on $t$ vertices with edge density at least $\alpha+c$.
\end{definition}

\noindent Erd\H{o}s observed that all values in $[0,1)$ are  jumps for $2$ and all values in $[0,\frac{r!}{r^r})$ are jumps for $r$.
He asked whether all values in $[0,1)$ are jumps for any $r\geq 2$--this was known as the jumping constant conjecture.
The question was answered negatively by Frankl and R\"odl in 1984 \cite{FR84},
who showed that $1-\frac{1}{l^{r-2}}$ is a non-jump for every $r\geq 3$ and $l>2r$. 
Since then, several pairs $(\alpha,r)$ of jumps/non-jumps have been identified \cite{baber2, FrPeRoTa07}.

In this paper, we consider the jump problem for general hypergraphs.
A hypergraph $H$ is a pair $(V,E)$ where $V$ is the vertex set and $E\subseteq 2^V$ is the edge set.
The set of edge sizes (or edge types) of $H$ is denoted by $R(H) := \{|e| : e \in E\}$.
If $R(H) = \{r\}$ then $H$ is an $r$-uniform hypergraph, or $r$-graph for short.
Given a fixed set $R$, we say a hypergraph $H$ is an \textbf{$R$-graph} if $R(H)\subseteq R$.
Notationally, we use the superscript $R$ to denote the edge types and the subscript $n$ to denote the number of vertices.
For example, $G_n^R$ denotes an $R$-graph $G$ on $n$ vertices; we may drop the $R$ or $n$ when either is understood under context.

In our previous paper \cite{tnuh}, we introduced the use of the Lubell function to measure the edge-density of a non-uniform hypergraph.
If $H$ is an $R$-graph on $n$ vertices, then the Lubell function of $H$ is the following: $h_{n}(H)=\sum_{e\in E(H)} \frac{1}{{n \choose |e| }}$.
The Lubell function is widely used in the study of extremal poset problems \cite{AxeManMar, BHLL, Bukh, GriLi, GriLu, GriLiLu, crowns}. 
Using the Lubell function, we defined the Tur\'an density of a family of $R$-graphs $\mathcal{H}$ is defined as $\pi(\mathcal{H}) := \lim_{n\to\infty} \pi_n(\mathcal{H})$
where $\pi_n (\mathcal{H})$ is the maximum of $h_n(G)$ among $\mathcal{H}$-free $R$-graphs $G$ on $n$ vertices.
We may occasionally write $\pi^{R}(\mathcal{H})$ when we emphasize that $R(\mathcal{H})$ may be a proper subset of $R$.
We proved that this limit always exists.
When $\mathcal{H}$ is a family of $r$-uniform hypergraphs the definition is the same as the classical definition.
Additionally, we demonstrated that there is a natural non-uniform generalization of the supersaturation lemma.
Because supersaturation is also observed in non-uniform hypergraphs, we also prove that blowing-up a non-uniform hypergraph doesn't change its Tur\'an density.
We also determine all the Tur\'an densities of a single $\{1, 2\}$-hypergraph.
These densities are always one of the following numbers
$$\bigg\{1,\frac{9}{8}, \frac{5}{4}, \frac{3}{2}, \frac{5}{3}, \ldots, 2-\frac{1}{k},\ldots \bigg \}$$
provided that the $\{1,2\}$-hypergraph is not also a $\{1\}$-graph or a $\{2\}$-graph.

The question of whether hypergraphs jump makes sense for $R$-graphs as well.
Using the Lubell function $h_n(H)$ as the measure of the edge density of the hypergraph $H$, one can easily extend the definition of jump to $R$-graphs.

\begin{definition}
	The value $\alpha\in [0, |R|]$ is a \textbf{jump} for $R$ if there
	exists a $c>0$ such that for every $\epsilon>0$ and every $t\geq \max \{r:r\in R\}$
	there exists an integer $n_0:=n_{0}(R, \alpha, \epsilon, t)$
	such that if $n\geq n_{0}$ and $G_{n}$ is an $R$-graph on $n$
	vertices with $h_{n}(G_{n})\geq \alpha+\epsilon$ then
	there exists a subgraph $H_t$ of $G_n$ on $t$ vertices with
	$h_{t}(H_{t})\geq \alpha+c$.
\end{definition}

\noindent Knowing whether a value $\alpha$ is a jump or not for $r$ rarely gives much information regarding Tur\'an densities.
One notable exception is the following example.
Suppose that one is trying to determine the Tur\'an density of a family of $R$-graphs $\mathcal{H}$.
It is known that $\alpha$ is a jump for $R$ with jumping constant $c$.
If it can be shown that $\pi(\mathcal{H})<\alpha+c$ then the fact that $\alpha$ is a jump implies that $\pi(\mathcal{H})\leq \alpha$.
This observation may be particularly helpful if when using a method such as Razborov's flag algebras.

In this paper, we introduce a new concept called strong jumps.

\begin{definition}
	The value $\alpha\in [0, |R|]$ is a \textbf{strong jump} for $R$ if there
	exists $c>0$ such that for every $t\geq	\max \{r:r\in R\}$ there exists an integer $n_0:=n_{0}(R, \alpha, c, t)$
	such that if $n\geq n_{0}$ and $G_{n}$ is an $R$-graph on $n$ vertices with $h_{n}(G_{n})\geq \alpha-c$ then
	there exists a subgraph $H_t$ of $G_n$ on $t$ vertices with $h_{t}(H_{t})\geq \alpha+c$.
\end{definition}

\noindent If a value $\alpha$ is a strong jump for $R$ it is also a jump for $R$;
the converse statement is not true.
A value $\alpha$ is a {\em weak jump} for $R$ if it is a jump but not strong jump. 
Refining the notion of jumps in this way turns out to have several nice consequences.
For example, the set of all strong jumps forms an open set (see Proposition \ref{open}).
Its complement, the set of not strong-jump values, has an algebraic structure (see Theorem \ref{t:add}) and is closely related to Tur\'an density.
We will show that $0$ is always a jump for $R$.
Furthermore, $0$ cannot be a strong jump; hence it is a weak jump.
Notice that $|R|$ is a weak jump for $R$; this is a degenerate case of the definition of jump.

\begin{theorem}\label{t:add}
	Consider any two finite sets $R_1$ and $R_2$ of non-negative integers.
	Suppose that $R_1\cap R_2=\emptyset$ and $R=R_1\cup R_2$. 
	If $\alpha_1$ is not a strong jump for $R_1$ and $\alpha_2$ is not a strong jump for $R_2$,
	then $\alpha_1+\alpha_2$ is not a strong jump for $R_1\cup R_2$.
\end{theorem}

The non-jump values can be determined by the following theorem.

\begin{theorem}\label{t:nonjump}
	For any fixed finite set $R$ of non-negative integers and $\alpha \in [0,|R|)$, 
	$\alpha$ is a non-jump if and only if it is the limit of a decreasing sequence of non-strong-jump values.
\end{theorem}

For example, the set of all strong jumps for $\{1\}$ forms an open interval $(0,1)$ while the only not strong jumps (both weak jumps) are $0,1$.
The set of all strong jumps for $\{2\}$ (i.e, graphs) are 

\[\cup_{k=1}^n\left(\frac{k-1}{k}, \frac{k}{k+1}\right);\]

\noindent while the non-strong-jumps (all are weak jumps) are 

\[0, \frac{1}{2},\frac{2}{3},\ldots, \frac{k}{k+1},\ldots, \text{ and } 1.\]

\noindent By Theorem \ref{t:add}, the following values are non-strong-jumps for $\{1,2\}$.

\[0, \frac{1}{2},\frac{2}{3},\ldots, \frac{k}{k+1}, \dots, 1, \frac{3}{2},\frac{5}{3},\ldots, \frac{2k+1}{k+1},2.\]

\noindent In this paper, we will determine all non-strong-jumps for $\{1,2\}$;
in particular, we will classify every $\alpha\in [0,2]$ as either a strong jump, weak jump, or non-jump.

\begin{theorem}\label{t:12}
	Every $\alpha\in[0,2]$ is a jump for $\{1,2\}$.
	Furthermore, the weak jumps are precisely the following:
	\[0, \frac{1}{2},\frac{2}{3},\ldots, \frac{k}{k+1}, \dots, 1, \frac{9}{8}, \frac{7}{6}, \ldots, 1+\frac{k}{4(k+1)},\ldots, \frac{5}{4}, \frac{3}{2},\frac{5}{3},\ldots, \frac{2k+1}{k+1},2.\]
\end{theorem}
It is interesting to observe that the interval $[0,2]$ has more non-strong-jumps than those values guaranteed by Theorem \ref{t:add}.
Those values $1+\frac{k}{4(k+1)}$ (for $k=1,2,3,\ldots$) are the Tur\'an density of a set of two $\{1,2\}$-hypergraphs.

\begin{definition}
	The \textbf{polynomial form} of a hypergraph $H^R_n$, denoted by $\lambda(P, \vec{x})$, is defined as
		\[\lambda(P, \vec{x}):= \sum_{e\in E(P)} |e|!\prod_{i\in e} x_{i}.\]
\end{definition}

\noindent The \textbf{Lagrangian} of $H$, denoted by $\lambda(H)$, is the maximum value of the polynomial $\lambda(P, \vec{x})$
over the simplex $S_{n}=\{(x_1,x_2,\ldots, x_n)\in [0,1]^n \colon \sum_{i=1}^nx_i=1\}.$
This version of Lagrangian of
non-uniform hypergraphs (and its weighed generalization) 
was also considered recently by Peng et.~al. \cite{TPZZ}, where they obtained
 some Motzkin-Straus type results in terms of the Lagrangian of non-uniform hypergraphs whose edges contain 1, or 2, and more vertices.

The definition of Lagrangian,
 when restricting to $r$-uniform hypergraphs, differs only by a multiplicative factor, $r!$, from the classical definition.
This does not substantively affect the results.  
In addition, many results regarding the Lagrangian of uniform hypergraphs generalize nicely.
For example, we have the following theorem, which generalizes a theorem due to Frankl and R\"odl for $r$-uniform hypergraphs.

\begin{theorem}\label{t:equiv}
	Let $R$ be a finite set of positive integers, and let $\alpha\in [0, |R|)$.
	Then $\alpha$ is a jump for $R$ if and only if there exists a finite family of $R$-graphs $\mathcal{F}$ such that
	\begin{enumerate}[(i)]
		\item $\pi(\mathcal{F})\leq \alpha$ and
		\item $\displaystyle \min_{F\in \mathcal{F}} \lambda(F) > \alpha$
	\end{enumerate}
	Moreover, $\alpha$ is a strong jump if the condition (i) is replaced by 
	\begin{enumerate}[(i')]
		\item $\pi(\mathcal{F})< \alpha$. 
	\end{enumerate}
\end{theorem}

\begin{theorem} \label{turan}
	Let $R$ be a finite set of positive integers.
	If $\mathcal{F}$ is a family (finite or infinite) of $R$-graphs then $\pi^{R}(\mathcal{F})$ is not a strong jump for $R$.
\end{theorem}

The paper is organized as follows.
In section 2, we will review some old facts and prove some new facts about the Tur\'an density of a family of non-uniform hypergraphs.
In section 3, we will prove theorem \ref{t:add}, \ref{t:nonjump}.
In section 4, we will study the properties of Lagrangians and prove
theorem \ref{t:equiv}. 
We will apply these results to $\{1,2\}$-graphs and prove theorem
\ref{t:12} in section 5. The relation between the non-strong-jump
values and the density of hereditary properties are studied in last section.


\section{Tur\'an Density of Hypergraphs}
The Tur\'an densities of $r$-uniform hypergraphs are the classical areas
of extremal graph theory and the readers are referred to Keevash's
survey paper \cite{KeevashSurvey}.
The Tur\'an densities have been generalized to a family of non-uniform hypergraphs in \cite{tnuh} and several properties were proved there. 
The key idea is using the Lubell function as the measurement of the density of non-uniform edges.

\begin{definition}
	Let $H=(V,E)$ be a hypergraph on $n$ vertices and let $\vec{s}=(s_1,...,s_n)$ be a vector of non-negative integers.
	Then the \textbf{blow-up} of $H$, denoted by $H(\vec{s})$, is
        the hypergraph obtained by replacing each vertex $v_{i}$ of
        $H$ with a set $V_{i}$ of vertices with $|V_{i}|=s_{i}$
 (where all $V_i$'s are disjoint).
	Each edge $e=\{v_{i_1},...,v_{i_k}\}\in E$ is replaced by the
        set of edges $V_{i_{1}}\times V_{i_{2}}\times...\times
        V_{i_k}$, so that
	\begin{itemize}
		\item $V(H(\vec{s})) = \bigcup_{i=1}^{n} V_{i}$ with $|V_{i}|=s_{i}$ and
		\item $E(H(\vec{s})) = \bigcup_{e\in E} \prod_{i\in e} V_{i}$.
	\end{itemize}
	If $\vec{s}=(s,s,...,s)$ then we write $H(s)=H(\vec{s})$.
\end{definition}

For a family of $R$-graphs $\cal H$, we define ${\cal
  H}(s)=\{H(s)\colon H\in {\cal H}\}$.

In \cite{tnuh} the following Lemma and Theorem were proved.
Note that we use $v(H)$ to denote $|V(H)|$.

\begin{lemma*}[Supersaturation]
	Let $\mathcal{H}$ be a finite family of hypergraphs.
	For any constant $\epsilon>0$ there exist positive constants $b$ and $n_{0}$ so that if $G$ is a hypergraph on $n>n_{0}$ vertices with
	$R(G)\subseteq R(\mathcal{H})$ and $h_{n}(G)>\pi(\mathcal{H})+a$ then
	$G$ contains at least $b\binom{n}{v(H)}$ copies of some $H\in \mathcal{H}$.
\end{lemma*}

\begin{theorem*}[Blow-ups]
	Let $\mathcal{H}$ be a finite family of hypergraphs and let $s\geq 2$.
	Then $\pi(\mathcal{H}(s))=\pi(\mathcal{H})$.
\end{theorem*}

If $H$ is contained in a blow-up of $G$, then we say that $H\leq G$.
The following are useful corollaries to the previous theorem.

\begin{corollary}
If $G$ contains $H$ and $G\leq H$ then $\pi(G)=\pi(H)$.
\end{corollary}

\begin{proof}
This follows from the observation that $\pi(H)\leq \pi(G)\leq \pi(H(s))=\pi(H)$.
\end{proof}

Baber and Talbot \cite{baber} point out the following useful Corollary for uniform hypergraphs.
The same thing is true for non-uniform hypergraphs.

\begin{corollary}
	If $G\leq H$ and $\mathcal{F}$ is a finite family of hypergraphs, then
	\begin{enumerate}[(i)]
		\item $\pi(\mathcal{F}\cup \{G\})\leq \pi(\mathcal{F}\cup \{H\})$,
		\item $\pi(\mathcal{F}\cup \{G\})=\pi(\mathcal{F}\cup \{G, H\})$.
	\end{enumerate}
\end{corollary}

\begin{proof}
	Let $s\geq 1$ be such that $G\subseteq H(s)$.
	The Blow-up theorem implies that $\pi(\mathcal{F}\cup
        \{H(s)\})=\pi(\mathcal{F}\cup 
\{H\})$.
	And $G\subseteq H(s)$ implies that $\pi(\mathcal{F}\cup \{G\})\leq \pi(\mathcal{F}\cup \{H(s)\})$.
	Thus $(i)$ holds.

	It is clear that $\pi(\mathcal{F}\cup \{G\})\geq \pi(\mathcal{F}\cup \{G, H\})$.
	Also, note that $\pi(\mathcal{F}\cup \{G\})=\pi(\mathcal{F}\cup \{G\}\cup \{G\})$.
	And $(i)$ implies that $\pi(\mathcal{F}\cup \{G\}\cup \{G\})\leq \pi(\mathcal{F}\cup \{G\}\cup \{H\})$.
	This gives us $(ii)$, the desired equality.
\end{proof}

We may occasionally wish to consider families of hypergraphs which are not finite.
Let $R$ be a finite set of non-negative integers.
Let $\mathcal{H}$ be a family of finite $R$-graphs.
Let $\mathcal{H}_{n}=\{H\in \mathcal{H}: |H|\leq n\}$.
Note that $\mathcal{H}_{n}$ is a finite family of $R$-graphs, so $\pi^{R}(\mathcal{H}_{n})$ is well defined.
Furthermore, if $m>n$ then $\pi^{R}(\mathcal{H}_{n})\geq \pi^{R}(\mathcal{H}_{m})$.
Thus $\lim_{n\to\infty} \pi^{R}(\mathcal{H}_{n})$ exists.

\begin{prop}
	If $\mathcal{H}$ is a family of finite $R$-graphs, then
	$\displaystyle \lim_{n\to\infty}\pi^{R(\mathcal{H})}(\mathcal{H}_{n})=\pi^{R}(\mathcal{H})$.
\end{prop}

\begin{proof}
	We only need to consider the case when $R=R(\mathcal{H})$.
	First, for any fixed value of $n$ we have that $\mathcal{H}_{n}\subseteq \mathcal{H}$.
	This implies that $\pi^{R}(\mathcal{H}_{n})\geq \pi(\mathcal{H})$ for every $n$.
	Hence $\lim_{n\to\infty}\pi^{R}(\mathcal{H}_{n})\geq \pi(\mathcal{H})$.

	Assume, towards a contradiction, that $\lim_{n\to\infty}\pi^{R}(\mathcal{H}_{n})=\pi(\mathcal{H})+\epsilon$ for some $\epsilon>0$.
	Then, for any $m,n \in \mathbb{N}$ there exists an $R$-graph $G_{m}$ on $m$ vertices that is $\mathcal{H}_{n}$-free satisfying $h_{m}(G_{m})=\pi(\mathcal{H})+\frac{\epsilon}{2}$.
	If $m$ is large enough, this implies that $G_{m}$ is not $\mathcal{H}$-free.
	Consider the case when $n=m$ and $m$ is large enough that $G_{m}$ is not $\mathcal{H}$-free.
	Since $G_{m}$ is not $\mathcal{H}$-free, there exists an $H\in \mathcal{H}$ such that $H\subseteq G_{m}$.
	This implies that $|H|\leq m$ and therefore $H\in \mathcal{H}_{m}$.
	Since $H\in \mathcal{H}_{m}$ and $H\in G_{m}$ we have that $G_{m}$ is not $H_{m}$-free.  Contradiction.
\end{proof}

There is also an analogue of supersaturation for infinite families.
In the proof of the lemma, the $b>0$ depends on $|\mathcal{H}|$ (a finite family) and $\epsilon$.
In the event that $|\mathcal{H}|$ is infinite, one can replace $|\mathcal{H}|$ in the theorem with $|\mathcal{H}_{m}|$
for the smallest $m$ such that $\pi^{R}(\mathcal{H}_{m})<\pi(\mathcal{H})+\epsilon$.

We also make the following observation:

\begin{prop}
	If $R=R_{1}\cup R_{2}$ and $R_{1}\cap R_{2}=\emptyset$ then for any $R_{1}$-graph $H_1$ and any $R_{2}$-graph $H_{2}$ we have that $\pi^{R}(H_1\cup H_2)=\pi^{R_{1}}(H_1)+\pi^{R_{2}}(H_2)$.
\end{prop}


\section{Strong Jumps, Weak Jumps, and Non-jumps}

The jumping constant conjecture, or the jump problem, for $r$-uniform hypergraphs was originated by Erd\H{o}s, who asked whether $r$-uniform hypergraphs jump.
Erd\H{o}s observed that every value $\alpha\in [0,1)$ is a jump for $R=\{2\}$ (i.e. graphs).
This is based on the following celebrated theorem due to Erd\H{o}s-Stone-Siminovits.

\begin{theorem}[Erd\H{o}s-Siminovits-Stone]
	For a non-empty graph $H$, a positive integer $t$, and $\epsilon>0$,
	there exists an $n_0:=n_0(H,t,\epsilon)$ such that if $G_n$ is any graph
	on $n\geq n_0$ vertices with at least $(1-\frac{1}{\chi(H)-1}+\epsilon){n\choose 2}$ edges, 
	then $G_{n}$ contains the blow-up graph $H(t)$ as a subgraph.  
\end{theorem}

\noindent Here $\chi(H)$ is the chromatic number of $H$ and $H(t)$ is a blowup of $H$.

For any $\alpha\in [0,1)$, let $k$ be an integer satisfying $\alpha\in [1-\tfrac{1}{k}, 1-\tfrac{1}{k+1})$.
One can choose $c=1-\frac{1}{k+1}-\alpha$.
Applying the Erd\H{o}s-Stone-Simonovits theorem, for any $t\geq 2$, any $\epsilon>0$, and sufficiently large $n$,
any graph on $n$ vertices with at least $(\alpha+\epsilon)\binom{n}{2}$ will contain $K_{k+1}(s)$, the blowup of the complete graph $K_{k+1}$.
Here we choose $s$ large enough so that $s(k+1)>t$.
Observe that there is an induced subgraph of $K_{k+1}(s)$ with density at least $1-\frac{1}{k+1}$, the density of $K_{k+1}(s)$.
Thus, $\alpha$ is a jump for $2$.

For $r$-uniform hypergraphs, Erd\H{o}s observed that every value in $[0,\frac{r!}{r^r})$ is a jump.
Here $\frac{r!}{r^r}$ is the density of $K^r_{s,s,\ldots, s}$, the complete $r$-partite $r$-graph of equal part size.
Based on these facts, Erd\H{o}s asked whether hypergraphs always jump.
In addition, Erd\H{o}s \cite{Er79}  put \$500 prize on whether $\tfrac{2}{9}$ is a jump for $3$-graphs.
Frankl and R\"odl \cite{FR84} showed  $1-\frac{1}{l^{r-2}}$ is a non-jump for every $r\geq 3$ and $l>2r$.
In their paper, Frankl and R\"odl introduced an equivalent definition of jump using admissible sequences
and upper densities (for $r$-uniform hypergraphs).
Now we generalize it to $R$-graphs.

\begin{definition}
	Let $\mathbf{G}:=\{G^R_{n_i}\}_{i=1}^\infty$ be a sequence of $R$-graphs.
	We say that $\mathbf{G}$ is an \textbf{admissible} sequence if $n_i\to \infty$ as $i\to\infty$ and 
	$\displaystyle \lim_{i\to\infty} h_{n_i}(G^R_{n_i})$ exists.
\end{definition}

\noindent The limit $\lim_{i\to\infty} h_{n_i}(G_{n_i}^R)$, denoted by $h(\mathbf{G})$, is called the \textbf{density} of the sequence $\mathbf{G}$.
Note $0\leq h(\mathbf{G})\leq |R|$ holds for any $R$-admissible sequence $\mathbf{G}$.
The converse also holds: for any $\alpha\in [0,|R|]$ there exists an $R$-admissible sequence $\mathbf{G}$ with density $\alpha$.

\begin{definition}
	The \textbf{upper density} of an admissible sequence of $R$-graphs $\mathbf{G}$, denoted by $\bar h(\mathbf{G})$, is defined as
	$\displaystyle \lim_{t\to\infty} \sigma_{t}(\mathbf{G})$, where $\displaystyle \sigma_{t}(\mathbf{G}):=\sup_i \sup_{T\in {[n_i]\choose t}}\{h_t(G_{n_i}^R[T])\}$
	is the supremum of the density of all induced subgraphs on $t$ vertices among all the graphs in the sequence.
\end{definition}

Note for any $t\geq \max\{r\colon r\in R\}$, $\sigma_{t}(\mathbf{G})$ is a decreasing function on $t$.
Thus, the limit, $\lim_{t\to\infty} \sigma_{t}(\mathbf{G})$ exists.
We also note that $\sup$ can be replaced by $\max$ in the definition.

\begin{lemma}\label{l:jump}
	A value $\alpha\in [0,|R|]$ is a jump for $R$ if and only if there exists a constant $c:=c(\alpha)>0$
	such that if $\mathbf{G}$ is an admissible sequence of $R$-graphs with $h(\mathbf{G})>\alpha$,
	we have that $\bar h(\mathbf{G}) \geq \alpha +c$.
\end{lemma}

\begin{proof}
	We first prove that it is necessary.
	Suppose that $\alpha$ is a jump for $R$. There is a constant $c>0$
	such that for any $\epsilon>0$ and any integer $t\geq \max\{r\colon r\in R\}$,
	there is an integer $n_0$ such that if $n\geq n_{0}$ and $G_{n}$ is an $R$-graph on $n$
	vertices with $h_{n}(G_{n})\geq \alpha+\epsilon$ then there exists a subgraph $H_t$ of $G_n$ on $t$ vertices with
	$h_{t}(H_{t})\geq \alpha+c$.

	Consider any admissible sequence of $R$-graphs $\mathbf{G}:=\{G^R_{n_i}\}_{i=1}^\infty$ with $h(\mathbf{G})>\alpha$.
	Choose $\epsilon=\frac{h(\mathbf{G})-\alpha}{2}$.
	There exists an $i_0$ such that $h_{n_i}(G_{n_i})>h(\mathbf{G})-\epsilon=\alpha +\epsilon$ for all $i\geq i_0$.
	Since $\alpha$ is a jump, for any $t$, for sufficiently large $i$, $G_{n_i}$ contains a subgraph $H_t$ with $h_t(H)\geq \alpha+c$.
	Thus, we have $\bar h(\mathbf{G}) \geq \alpha +c$.

	We now prove the contrapositive of the reverse implication.
	Suppose that $\alpha$ is not a jump.
	For any $c>0$, there exist $\epsilon>0$ and $t>\max\{r\colon r\in R\}$, so that for any $i$ there exists a graph $G_{n_i}$ with $n_i\geq i$
	satisfying $h_{n_i}(G_{n_i})\geq \alpha + \epsilon$ and $G_{n_i}$ contains no subgraph $H_t$ with density $h_t(H_t)\geq \alpha+c$.
	The sequence $\mathbf{G}$ formed by the $G_{n_i}$'s may not be admissible because the limit $\lim_{i\to\infty} h_{n_i}(G_{n_i})$ may not exist.
	However, by deleting some edges, each $G_{n_i}$ contains a spanning subgraph $G'_{n_i}$ with $h_{n_i}(G'_{n_i})=\alpha +\epsilon \pm O(\frac{1}{n})$.
	Now $\mathbf{G'}:=\{G'_{n_i}\}_{i=1}^\infty$ is an admissible sequence with $h(\mathbf{G'})=\alpha+\epsilon>\alpha$.
	Note $G'_{n_i}$ contains no subgraph $H_t$ with density $h_t(H_t)\geq \alpha+c$.
	Thus we have $\bar h(\mathbf{G'})< \alpha +c.$
\end{proof}

The structure of the set of all jump values is not obvious. 
However, the set of strong jumps has very nice structure.

\begin{prop}\label{open}
	For any fixed finite set $R$ of non-negative integers, the set of all strong jumps for $R$ is an open subset of $(0,|R|)$. 
\end{prop}

\begin{proof}
	Suppose $\alpha$ is a strong jump. 
	Let $c>0$ be the positive constant (from the definition) whose existence is guaranteed by the fact that $\alpha$ is a strong jump.
	For every $\beta\in (\alpha-\frac{c}{2}, \alpha+\frac{c}{2})$, we can choose a new constant $c':=\frac{c}{2}$.
	For every $t\geq \max \{r:r\in R\}$ and every $R$-graph $G_{n}$ with $n\geq n_0$
	and $h_{n}(G_{n})\geq \beta-c' >\alpha-c$, there exists a subgraph $H_t$ of $G_n$ on $t$ vertices with
	$h_{t}(H_{t})\geq \alpha+c>\beta+c'$.
	Thus $\beta$ is a strong jump. 
	Hence the set of strong jumps is open.
\end{proof}

\begin{lemma}\label{l:strongjump}
	A value $\alpha\in [0,|R|)$ is a strong jump for $R$ if and only if
	there exists a constant $c:=c(\alpha)>0$ such that every admissible sequence of $R$-graphs $\mathbf{G}$
	with $h(\mathbf{G})=\alpha$ has upper density $\bar h(\mathbf{G})\geq \alpha +c$.
\end{lemma}

\begin{proof}
	This is similar to the proof of Lemma \ref{l:jump}.
	Suppose that $\alpha$ is a strong jump for $R$. 
	There is a constant $c>0$ such that for any integer $t\geq \max\{r\colon r\in R\}$, there is an integer $n_0$
	such that if $n\geq n_{0}$ and $G_{n}$ is an $R$-graph on $n$ vertices with $h_{n}(G_{n})\geq \alpha-c$ then
	there exists a subgraph $H_t$ of $G_n$ on $t$ vertices with $h_{t}(H_{t})\geq \alpha+c$.

	Consider any admissible sequence of $R$-graphs $\mathbf{G}:=\{G^R_{n_i}\}_{i=1}^\infty$ with $h(\mathbf{G})=\alpha$.
	There exists an $i_0$ such that $h_{n_i}(G_{n_i})>h(\mathbf{G})-c$ for all $i\geq i_0$.
	Since $\alpha$ is a strong jump, for any $t$, for sufficiently large $i$, $G_{n_i}$ contains a subgraph $H_t$ with $h_t(H)\geq \alpha+c$.
	Thus, we have $\bar h(\mathbf{G}) \geq \alpha +c$.

	Now we prove the contrapositive of the reverse implication.
	Assume that $\alpha$ is not a strong jump.
	For any $c>0$, there exists a $t>\max\{r\colon r\in R\}$, and for any $i$ there exists a graph $G_{n_i}$
	with $n_i\geq i$ satisfying $h_{n_i}(G_{n_i})\geq \alpha -c$ and $G_{n_i}$ contains no subgraph $H_t$
	with density $h_t(H_t)\geq \alpha+c$.
	In particular, for $k=1,2,\ldots$, we choose $c=\frac{1}{k}$ and $i=k$.
	We obtain a sequence of $R$-graphs $\mathbf{G}=\{G_{n_k}\}_{k=1}^\infty$.
	By deleting some edges, $\mathbf{\bf G}$ contains a spanning subgraph $\{G'_{n_k}\}$ with $h_{n_k}(G'_{n_k})=\alpha -\frac{1}{k} \pm O(\frac{1}{n})$.
	Now ${\bf G'}:=\{G'_{n_k}\}_{k=1}^\infty$ is an admissible sequence with $h(\mathbf{G'})=\alpha$.
	Note $G'_{n_i}$ contains no subgraph $H_t$ with density $h_t(H_t)\geq \alpha+ \frac{1}{k}$ for $i$ sufficiently large.
	We have $\bar h(\mathbf{G'})\leq \alpha.$
\end{proof}

\begin{corollary}\label{nsj}
	The following statements are equivalent.
	\begin{enumerate}
		\item An value $\alpha\in [0,|R|]$ is NOT a strong jump for $R$.
		\item There exists an admissible sequence of $R$-graphs ${\bf G}:=\{G_{n_i}\}_{i=1}^\infty$ satisfying
			$h(\mathbf{G})=\bar h(\mathbf{G})= \alpha$.
		\item For a given increasing sequence of positive integers $n_1<n_2<\ldots$, there exists an admissible 
			sequence of $R$-graphs $\mathbf{G}:=\{G_{n_i}\}_{i=1}^\infty$ satisfying $h(\mathbf{G})=\bar h(\mathbf{G})=\alpha$.
	\end{enumerate}
\end{corollary}

\begin{proof}
	(1) $\Rightarrow$ (2): See the proof of Lemma \ref{l:strongjump}.

	(2) $\Rightarrow$ (3): Suppose that there an admissible sequence of $R$-graphs $\mathbf{G}:=\{G_{n'_i}\}_{i=1}^\infty$ satisfying 
	$h(\mathbf{G})=\bar h(\mathbf{G})= \alpha$.
	For each $i=1,2,\ldots$, find an index $n'_j>n_i$ so that $h_{n_j}(G_{n'_j})>\alpha-\frac{1}{i}$.
	There is subgraph of $G_{n'_j}$ on $n_i$ vertices whose density is at least $h_{n_j}(G_{n'_j})>\alpha-\frac{1}{i}$.
	By deleting some edges if necessary, there exists an subgraph $G'_{n_i}\subset G_{n'_j}$ satisfying
	$h_{n_i}(G'_{n_i})=\alpha-\frac{1}{i}+O(\frac{1}{n_i})$.
	Let $\mathbf{G'}:=\{G_{n_i}\}_{i=1}^\infty$. 
	We have 
	\[h(\mathbf{G'})=\lim_{i\to \infty}h_{n_i}(G'_{n_i})=\alpha,\] and \[\bar h(\mathbf{G'})\leq \bar h(\mathbf{G})=\alpha.\]
	Since $h(\mathbf{G'})\leq \bar h(\mathbf{G'})$, we have $\bar h(\mathbf{G'})=h(\mathbf{G'})=\alpha.$

	(3) $\Rightarrow$ (1): This is the contrapositive of Lemma \ref{l:strongjump}.
\end{proof}

In the introduction, we stated the following theorem for which we will now give the proof.

\begin{theorem*}
	Consider any two finite sets $R_1$ and $R_2$ of non-negative integers.
	Suppose that $R_1\cap R_2=\emptyset$ and $R=R_1\cup R_2$. 
	If $\alpha_1$ is not a strong jump for $R_1$ and $\alpha_2$ is not a strong jump for $R_2$,
	then $\alpha_1+\alpha_2$ is not a strong jump for $R_1\cup R_2$.
\end{theorem*}

\begin{proof}
	For $j\in \{1,2\}$, since $\alpha_{j}$ is not a strong jump for $R_j$, by corollary \ref{nsj},
	there exists an admissible $R_{j}$ sequence of graphs $\mathbf{G}^{R_j}:=\{G_{n}^{R_j}\}_{n=1}^\infty$
	satisfying $h(\mathbf{G}^{R_j})=\bar  h(\mathbf{G}^{R_j})=\alpha_j$. 

	For $n=1,2,3,\ldots$, construct a new sequence of graphs ${\bf H}^R:=\{H^R_{n}\}_{n=1}^\infty$ as follows.
	The vertex set of $H^R_n$ is the common vertex set $[n]$ while the edge set of $H^R_n$ is the union of $E(G^{R_1}_{n})$ and $E(G^{R_2}_{n})$. 
	Since $R_{1}\cap R_{2}=\emptyset$, we have 
		\begin{align*}
			h(\mathbf{H}) &= h(\mathbf{G}^{R_1})+h(\mathbf{G}^{R_2}) = \alpha_1+\alpha_2 \\
        	\bar h(\mathbf{H}) &\leq \bar h(\mathbf{G}^{R_1}) + \bar h(\mathbf{G}^{R_2})=\alpha_1+\alpha_2.
		\end{align*}
	Since $h(\mathbf{H}) \leq \bar  h(\mathbf{H})$, we have $\bar h(\mathbf{H})=\alpha_1+\alpha_2$.
	Hence $\alpha_{1}+\alpha_{2}$ is not a strong jump for $R_{1}\cup R_{2}$.
\end{proof}

\begin{lemma}
	If $\alpha=\bar{h}(\mathbf{G})$ for some $R$-admissible sequence $\mathbf{G}$, then $\alpha$ is not a strong jump for $R$.
\end{lemma}

\begin{proof}
	It suffices to find an $R$-admissible sequence $\mathbf{F}$ such that $h(\mathbf{F})=\bar{h}(\mathbf{F})=\alpha$.
	For each $t\geq \max\{r:r\in R\}$ there is some $t$-subset $T$ of vertices of $G_{i_{t}}$ such that $h_{t}(G_{i_{t}}[T])=\sigma_{t}(\mathbf{G})$.
	Let $F_{t}=G_{i_{t}}[T]$ and create the $R$-admissible sequence $\mathbf{F}=\{F_{t}\}$.
	By construction $\lim_{t\to\infty} h_{t}(F_{t})=\lim_{t\to\infty}\sigma_{t}(\mathbf{G})=\alpha$.
	Furthermore, $\bar{h}(\mathbf{F})=h(\mathbf{F})=\alpha$.
	Hence $\alpha$ is not a strong jump.
\end{proof}

\begin{lemma}
	If for some $c>0$ every value in the interval $(\alpha,\alpha+c)$ is a strong jump for $R$, then $\alpha$ is a jump for $R$.
\end{lemma}

\begin{proof}
	Consider any admissible sequence $\mathbf{G}$ with $h(\mathbf{G})>\alpha$.
	We need to show that there exists a constant $c^{\prime}>0$ such that $\bar h(\mathbf{G})\geq \alpha+c^{\prime}$.
	Take $c^{\prime}=c>0$.
	If $h(\mathbf{G})\geq \alpha+c$, then we are done since $\bar h(\mathbf{G})\geq h(\mathbf{G})$.
	Otherwise, $h(\mathbf{G})$ is in the interval $(\alpha, \alpha+c)$.
	By hypothesis, $h(\mathbf{G})$ is a strong jump.
	Since $\bar{h}(\mathbf{G})\geq h(\mathbf{G})$ and $\bar{h}(\mathbf{G})$ is not a strong jump, we have that $\bar{h}(\mathbf{G})\geq \alpha+c$.
	Therefore, $\alpha$ is jump.
\end{proof}

In the introduction, we stated the following theorem for which we will now give the proof.

\begin{theorem*}
	For any fixed finite set $R$ of non-negative integers and $\alpha \in [0,|R|)$, 
	$\alpha$ is a non-jump if and only if it is the limit of a decreasing sequence of non-strong-jump values.
\end{theorem*}

\begin{proof}
	We first prove that it is sufficient.
	Suppose that there exists a sequence of non-strong-jump values $\alpha_1>\alpha_2>\ldots$ with $\lim_{i\to \infty}\alpha_i=\alpha$.
	We need to show $\alpha$ is not a jump $R$. 
	We prove it by contradiction.

	Assume that $\alpha$ is a jump for $R$.
	There exists a $c:=c(\alpha)>0$ so that any admissible sequence $\mathbf{G}$ with $h(\mathbf{G})>\alpha$ satisfies $\bar h(\mathbf{G})\geq \alpha+c$.
	Choose $i$ large enough so that $\alpha_i<\alpha+\frac{c}{2}$.
	Since $\alpha_i$ is not a strong jump, there exists admissible sequence $\mathbf{H}$ with $\bar h(\mathbf{H})=h(\mathbf{H})=\alpha_i$.
	Note $\alpha<\alpha_i<\alpha+c$. Contradiction.

	Now we prove that it is also necessary.
	It is follows from the previous lemma.
	Apply the lemma to construct the values $\alpha_k$.
	Let $\alpha_1=|R|$. We will construct $\alpha_k$ recursively.
	Since $\alpha$ is not jump, applying the lemma with $c_k:= \min\{\alpha_{k-1}-\alpha, \frac{1}{k}\})$, 
	there exists a non-strong-jump value $\alpha_k$ in $(\alpha, \alpha +c_k)$. 
	Clearly, we have $\alpha_1>\alpha_2>\cdots>\alpha_k>\cdots$ and $\lim_{k\to\infty}\alpha_k=\alpha$.
\end{proof}

Finally, we previously stated the following theorem for which we will now give the proof.

\begin{theorem*}
	Let $R$ be a finite set of positive integers.
	If $\mathcal{F}$ is a family (finite or infinite) of $R$-graphs then $\pi^{R}(\mathcal{F})$ is not a strong jump for $R$.
\end{theorem*}

\begin{proof}[Proof of Theorem \ref{turan}:]
	Let $\mathcal{H}$ be a family of finite $R$-graphs and $\alpha:=\pi(\mathcal{H})$.
	By the definition of Tur\'an density, there exists an admissible sequence of $\mathcal{H}$-free $R$-graphs $\mathbf{G}:=\{G^R_n\}_{n=1}^\infty$ with $h(\mathbf{G})=\alpha$.
	Observe any subgraph of $G^R_n$ is also $\mathcal{H}$-free.
	We have $\bar h(\mathbf{G})=\alpha$. 
	By corollary \ref{nsj}, $\alpha$ is not a strong jump.
\end{proof}


\section{The Lagrangian}

Frequently we need to describe a sequence of hypergraphs whose size grows to infinity.
We begin this section by giving a formal way of describing a family of hypergraphs, specifically, a hypergraph pattern.

\begin{definition}
	A hypergraph \textbf{pattern}, $P$, is a pair $P=(V,E)$.
	$V=\{v_1,...,v_{n}\}$ is a vertex set and the edge set $E$ is a finite set of multisets of vertices.
	A typical element $e\in E$ will have the form $e=\{k_1\cdot v_1, k_2\cdot v_2,...,k_{n}\cdot v_{n}\}$
	where $k_{i}$ is a non-negative integer for each $1\leq i\leq n$.
	We say that $|e|=\sum_{i=1}^{n}k_{i}$.
\end{definition}

\begin{definition}
	Suppose that $P$ is a hypergraph pattern on $n$-vertices and $m$-edges.
	Let $\vec{s}=(s_1,...,s_{n})$ be a non-negative vector of integers.
	A hypergraph $H=P(\vec{s})$ is a \textbf{realization} of a pattern $P$ if:
	\begin{itemize}
		\item $\displaystyle V(H)=\bigcup_{i=1}^{n} V_{i}$ with $|V_i|=s_{i}$ for $1\leq i\leq n$
		\item $\displaystyle E(H)=\bigcup_{e\in E(P)} \binom{V_{1}}{k_{1}}\times \binom{V_{2}}{k_{2}}\times ... \times \binom{V_{n}}{k_{n}}$
	\end{itemize}
\end{definition}

We view a realization of $P$ essentially as a blow-up of the pattern $P$.
Note that any hypergraph can also be viewed as a pattern--but not every pattern can be viewed as a hypergraph.

Let $P$ be a hypergraph pattern on $n$ vertices.
Suppose that we want a realization of $P$ with $N$ vertices.
We can choose a vector $\vec{x}\in S_{n}$ such that $x_{i}N\in \mathbb{Z}$ for each $i$.
Then $H=P(N\vec{x})$ is a realization of $P$ on $N$ vertices and $|V_{i}|=x_i N$ for each $1\leq i\leq n$.
Let $e=\{k_1\cdot v_1,...,k_{n}\cdot v_{n}\}$ be an edge in $E(P)$.
The edges of $H$ that correspond to $e$ contribute
\[\frac{\prod_{i=1}^{n} \binom{|V_i|}{k_i}}{\binom{N}{|e|}}=\frac{\prod_{i=1}^{n} \frac{(x_i N)^{k_i}}{k_i !}}{\frac{N^{|e|}}{|e|!}} + o(1) = \binom{|e|}{k_1,k_2,...,k_n}\prod_{i=1}^{n} x_{i}^{k_i} + o(1)\]
to the Lubell function of $H$.

\begin{definition}
	Let $P$ by a hypergraph pattern on $n$ vertices.
	The \textbf{polynomial form} of $P$, denoted by $\lambda(P, \vec{x})$, is defined as
	\[\lambda(P, \vec{x}):= \sum_{e\in E(P)} \binom{|e|}{k_1,k_2,...,k_n}\prod_{i=1}^{n} x_{i}^{k_i}.\]
\end{definition}

\noindent Note that $\lambda(P, \vec{x})$ can be viewed as a polynomial in $\mathbb{Z}[x_1,...,x_{n}]$ when the vector $\vec{x}$ is unknown,
or as a real number when $\vec{x}$ is specified.
Since $S_{n}$ is compact, it follows that (the polynomial) $\lambda(P, \vec{x})$ attains a maximum value on $S_{n}$.

\begin{definition}
	The \textbf{Lagrangian} or \textbf{blow-up density} of a hypergraph pattern $P$ is 
	\[\lambda(P):=\max_{\vec{x}\in S_{n}} \lambda(P, \vec{x}).\]
\end{definition}

Typically, the polynomial form (for $r$-uniform hypergraphs) is defined so that every term has coefficient 1
and the blow-up density is defined to be the largest edge density (in the limit) one can obtain by blowing up a given hypergraph.
For the $r$-uniform graphs the Lagrangian and blow-up density differ by a constant ($r!$).
When the graph is uniform, differing by a constant is easy to work around.
However, if we generalize the Lagrangian to non-uniform graphs and leave every term monic, then the blow-up density and the Lagrangian no longer differ by a constant.
This is unacceptable for the applications we have in mind.
For this reason, we adjust the coefficients of the polynomial form so that the value of the Lagrangian is meaningful.

The following proposition is the first reason for our interest in Lagrangians.

\begin{prop}
	Let $P$ be a hypergraph pattern on $n$ vertices.
	Let $\mathcal{F}$ be a family of $R$-graphs.
	If every realization $H$ of $P$ that is a hypergraph has the properties that
	$R(H)\subseteq R$ and $H$ is $\mathcal{F}$-free, then $\lambda(P)\leq \pi^{R}(\mathcal{F})$.
\end{prop}

For a vector $\vec{x}=(x_1,...,x_{n})$ we denote by $\supp(\vec{x})$, the support of $\vec{x}$, the set of indices $i$ such that $x_{i}\neq 0$.
Let $J\subseteq [n]$ be a set of indices.
Then $S_{J}=\{\vec{x}\in S_{n}:\supp(\vec{x})=J\}$.
When we refer to $S_{J}$ we will always assume that $J\neq \emptyset$ (otherwise $S_{J}=\emptyset$ since $\vec{0}\notin S_{n}$).
The following lemmas are generalizations of results due to Frankl and R\"odl.
The proofs are similar, in some case with no essential difference.
The one thing to keep in mind is that the way we have defined a Lagrangian differs slightly from the \textit{standard} definition,
primarily because one typically considers only uniform hypergraphs.

\begin{lemma}
Let $H$ be a hypergraph and suppose that $\vec{y}\in S_{J}$ satisfies $\lambda(H, \vec{y})=\lambda$ and
$|J|$ is minimal.  
Then for any $a, b\in J$ there exists an edge $e\in E(H)$ with $\{a,b\}\subseteq e\subseteq J$.
\end{lemma}

\begin{proof}
Towards a contradiction, assume that there is no edge satisfying $\{a,b\}\subseteq e \subseteq J$.
We will use $\vec{x}$ to denote variables, and $\vec{y}$ as an assignment of those variables.
Since there is no edge $e\subseteq J$ with $\{a,b\}\subseteq e$, it follows that
\[\frac{\partial^2}{\partial x_{a} \partial x_{b}}\lambda(H, \vec{y}) = 0.\]
Without loss of generality, assume that 
\[\frac{\partial}{\partial x_{a}}\lambda(H, \vec{y})\leq \frac{\partial}{\partial x_{b}}\lambda(H, \vec{y}).\]
Set $\delta = \min \{y_a, 1-y_b\}\geq 0$.
Create a new vector $\vec{z}$ as follows: $z_{a} = y_a - \delta \geq 0$, $z_{b}=y_b + \delta \leq 1$, and $z_{i}=y_{i}$ for every value of $i$.
Note that $\vec{z}\in S_{n}$ and $z_{i}=0$ if $i\notin J$ and $z_{a}=0$.
The last follows from the fact that if $\delta\neq y_{a}$ then $z_{b}=1$.
We will now show that $\lambda(H, \vec{z})\geq \lambda(H, \vec{y})=\lambda(H)$ contradicting the minimality of $|J|$.
\begin{align*}
\lambda(H, \vec{z}) &= \sum_{e\in H} |e|!\prod_{i\in e}z_{i} \\
                    &= \sum_{\stackrel{e\in H}{a,b\notin e}} |e|!\prod_{i\in e}z_{i} + \sum_{\stackrel{e\in H}{a\in e}}|e|!\prod_{i\in e}z_{i} + \sum_{\stackrel{e\in H}{b\in e}}|e|!\prod_{i\in e}z_{i} \\
		    &= \sum_{\stackrel{e\in H}{a,b\notin e}} |e|!\prod_{i\in e}y_{i} + z_{a}\sum_{\stackrel{e\in H}{a\in e}}|e|!\prod_{\stackrel{i\in e}{i\neq a}} y_{i} + z_{b}\sum_{\stackrel{e\in H}{b\in e}}|e|!\prod_{\stackrel{i\in e}{i\neq b}} y_{i} \\
		    &= \sum_{\stackrel{e\in H}{a,b\notin e}} |e|!\prod_{i\in e}y_{i} + (y_a-\delta)\sum_{\stackrel{e\in H}{a\in e}}|e|!\prod_{\stackrel{i\in e}{i\neq a}} y_{i} + (y_{b}+\delta)\sum_{\stackrel{e\in H}{b\in e}}|e|!\prod_{\stackrel{i\in e}{i\neq b}} y_{i} \\
		    &= \lambda(H, \vec{y}) + \delta\left(\sum_{\stackrel{e\in H}{b\in e}}|e|!\prod_{\stackrel{i\in e}{i\neq b}}y_{i}-\sum_{\stackrel{e\in H}{a\in e}}|e|!\prod_{\stackrel{i\in e}{i\neq a}}y_{i}\right) \\
		    &= \lambda(H, \vec{y}) + \delta\left(\frac{\partial}{\partial x_{b}} \lambda(H, \vec{y}) - \frac{\partial}{\partial x_{a}} \lambda(H, \vec{y})\right) \\
		    &\geq \lambda (H, \vec{y}).
\end{align*}
We have created a new optimal vector $\vec{z}$ with $|\supp(\vec{z})|<|\supp(\vec{y})|=|J|$.
Contradiction.
\end{proof}

\begin{definition}
	Let $H^{k}$ be a $k$-uniform hypergraph.
	We say that two vertices $i, j$ are equivalent if for every $e\in \binom{V(H)-\{i,j\}}{k-1}$
	it follows that $e\cup \{i\}\in E(H^{k})$ if and only if $e\cup \{j\}\in E(H^{k})$.
\end{definition}

\begin{definition}
	Let $H$ be a non-uniform hypergraph.
	We say that two vertices $i$ and  $j$ are equivalent if for every $k\in R(H)$, 
	$i$ and $j$ are equivalent in $H^{k}$.
\end{definition}

\begin{lemma}\label{l:vertequiv}
	Let $H$ be a hypergraph whose vertex set is $[n]$ and suppose that $a$ and $b$ are equivalent vertices.
	Then there exists a $\vec{y}\in S_{n}$ satisfying $\lambda(H, \vec{y})=\lambda(H)$ and $y_{a}=y_{b}$.
	Moreover, for any vector $\vec{y}\in S_{n}$ satisfying $\lambda(H, \vec{y})=\lambda(H)$ 
	if there exists an edge $e\in H$ such that $\{a, b\}\subseteq e\subseteq \supp(\vec{y})\cup \{a\}$
	then $y_{a}=y_{b}$.
\end{lemma}

\begin{proof}
	Suppose that $y_{a}\neq y_{b}$.
	Define a new vector $\vec{z}$ so that $z_{a}=z_{b}=\frac{y_{a}+y_{b}}{2}$ and $z_{v}=y_{v}$ otherwise.
	Clearly, $\vec{z}\in S_{n}$.
	We just need to check that $\lambda(H, \vec{z})\geq \lambda(H, \vec{y})$.
		\begin{align*}
			\lambda(H, \vec{z}) &= \sum_{e\in H} |e|! \prod_{i\in e} z_{i} \\
			                    &= \sum_{\stackrel{e\in H}{a,b\notin e}} |e|!\prod_{i\in e} y_{i} + \sum_{\stackrel{e\in H}{a\in e, b\notin e}} |e|!z_{a}\prod_{\stackrel{i\in e}{i\neq a}} y_{i} + \sum_{\stackrel{e\in H}{b\in e, a\notin e}} |e|!z_{b}\prod_{\stackrel{i\in e}{i\neq b}}y_{i} + \sum_{\stackrel{e\in H}{\{a,b\}\subseteq e}} |e|!z_{a}z_{b}\prod_{\stackrel{i\in e}{a,b\neq i}} y_{i} \\
							    &= \sum_{\stackrel{e\in H}{a,b\notin e}} |e|!\prod_{i\in e} y_{i} +2\left(\frac{y_{a}+y_{b}}{2}\right) \sum_{\stackrel{e\in H}{a\oplus b\in e}} |e|!\prod_{\stackrel{i\in e}{i\neq a, b}} y_{i} + \left(\frac{y_{a}+y_{b}}{2}\right)^{2} \sum_{\stackrel{e\in H}{\{a,b\}\subseteq e}} |e|!\prod_{\stackrel{i\in e}{a,b\neq i}} y_{i} \\
							    &\geq \lambda(H, \vec{y})
		\end{align*}
	since $\left(\frac{y_{a}+y_{b}}{2}\right)^2\geq y_{a}y_{b}$ (with equality if and only if $y_{a}=y_{b}$).
\end{proof}

We stated the following theorem in the introduction.
We will now give the proof.
It is a generalization of a theorem due to R\"{o}dl and Frankl \cite{FR84}.

\begin{theorem*}
	Let $R$ be a finite set of positive integers, and let $\alpha\in [0, |R|)$.
	Then $\alpha$ is a jump for $R$ if and only if there exists a finite family of $R$-graphs $\mathcal{F}$ such that
	\begin{enumerate}[(i)]
		\item $\pi(\mathcal{F})\leq \alpha$ and
		\item $\displaystyle \min_{F\in \mathcal{F}} \lambda(F) > \alpha$
	\end{enumerate}
	Moreover, $\alpha$ is a strong jump if the condition (i) is replaced by 
	\begin{enumerate}[(i')]
		\item $\pi(\mathcal{F})< \alpha$. 
	\end{enumerate}
\end{theorem*}

\begin{proof}
	First, let us suppose that $\alpha\in [0, |R|)$ is a jump for $R$.
	By definition, there exists some $\Delta>0$ so that for any $k$ and any $\epsilon>0$ there exists an $n_{0}(R, k, \epsilon)$ so that if $G$ is an $R$-graph on $n\geq n_{0}$
	vertices, with $h_{n}(G)\geq \alpha+\epsilon$ then $G$ contains a subgraph $H$ on $k$ vertices with $h_{k}(H)>\alpha+\Delta$.
	We will find a finite family of graph $\mathcal{F}$ with properties $(i)$ and $(ii)$ above.

	Suppose that $R=\{r_{1},...,r_{t}\}$ with $r_{1}<r_{2}<...<r_{t}$.
	Fix $k$ large enough that the constant 

	\[c=c(R):=\left(1-\frac{1}{k}\right)\left(1-\frac{2}{k}\right)...\left(1-\frac{r_{t}-1}{k}\right)>\frac{\alpha+\frac{\Delta}{2}}{\alpha+\Delta}.\]
	
	\noindent Let $\mathcal{F}$ be the set of all hypergraphs $F$ on exactly $k$ vertices satisfying the following two conditions:
	\begin{enumerate}[(i)]
		\item $R(F)\subseteq R$
		\item $\displaystyle \lambda\left(F, k^{-1}(1,1,...,1)\right)=\sum_{r\in R}\frac{r!}{k^{r}}|E^{r}(F)|\geq \alpha + \frac{\Delta}{2}$.
	\end{enumerate}
	Note that $\min_{F\in \mathcal{F}}\lambda(F)\geq \alpha+\frac{\Delta}{2}>\alpha$.
	It remains to show that $\pi(\mathcal{F})\leq \alpha$.

	Let $\epsilon>0$ be given.
	Let $G_{n}$ be a graph on $n\geq n_{0}(R, k, \epsilon)$ vertices (enough vertices) with $h_{n}(G)\geq \alpha+\epsilon$.
	We need to show that $G_{n}$ contains a member of $\mathcal{F}$.
	First, by hypothesis that $\alpha$ is a jump (and $G_{n}$ has enough vertices) we know that $G_{n}$ contains a graph $H_{k}$ on $k$-vertices
	with $h_{k}(H_{k})\geq \alpha+\Delta$.
	We will now show that $H_{k}\in \mathcal{F}$.
	First, it is clear that $R(H)\subseteq R$ since $R(G)\subseteq R$.
	\begin{align*}
		\alpha+\Delta &\leq h_{k}(H_{k}) \\
		              &= \sum_{r\in R} \frac{|E^{r}(H_{k})|}{\binom{k}{r}} \\
					  &= \sum_{r\in R} \frac{r!}{k(k-1)...(k-r+1)}|E^{r}(H_{k})| \\
					  &= \sum_{r\in R} \frac{r!}{k^{r}(1-\frac{1}{k})(1-\frac{2}{k})...(1-\frac{r-1}{k})}|E^{r}(H_{k})| \\
					  &\leq \frac{1}{c} \sum_{r\in R}\frac{r!}{k^{r}} |E^{r}(H_{k})|.
	\end{align*}
	
	\noindent Rearranging terms, we have that

	\[\sum_{r\in R}\frac{r!}{k^{r}}|E^{r}(H_{k})| \geq c(\alpha+\Delta) > \alpha+\frac{\Delta}{2}.\]

	\noindent Thus, $H_{k}$ is a member of $\mathcal{F}$.
	Hence $\pi(\mathcal{F})\leq \alpha$ as desired.

	Now, suppose that we have a finite family $\mathcal{F}$ with the properties that
	$R(\mathcal{F})\subseteq R$, $\min_{F\in \mathcal{F}}\lambda(F)>\alpha$, and $\pi(\mathcal{F})\leq \alpha$.
	We need to show that $\alpha$ is a jump for $R$.
	Write $R=\{r_1,...,r_{t}\}$ with $r_1<r_2<...<r_{t}$ and fix $\epsilon>0$ and $k\geq r_{t}$.
	Let $\Delta=\min_{F\in \mathcal{F}} \lambda(F)-\alpha>0$.
	Choose $n_{0}$ large enough that if $n\geq n_{0}$ it follows that each $\lambda(F)$ for $F\in \mathcal{F}$ can be approximated by some vector
	
	\[\frac{\vec{x}_{F}}{n}=\left(\frac{x_1}{n},...,\frac{x_{|F|}}{n}\right)\in S_{|F|}\]
	
	\noindent where each $\vec{x}\in \mathbb{N}^{|F|}$ satisfying $\lambda(F, \frac{\vec{x}_{F}}{n})\geq \alpha + \frac{\Delta}{2}$.
	Now, consider the family $\mathcal{F}^{\prime}$ where
	
	\[\mathcal{F}^{\prime}:=\{F(\vec{x}_{F}):F\in \mathcal{F}\}\]

	\noindent obtained by blowing-up each graph in $\mathcal{F}$ so as to maximize the Lubell value of each graph.
	Since $\mathcal{F}^{\prime}$ is obtained by blowing up graphs in $\mathcal{F}$, it follows that $\pi(\mathcal{F}^{\prime})=\pi(\mathcal{F})\leq \alpha$.
	Suppose that $G$ has $N$ vertices and $h_{N}(G)\geq \alpha+\epsilon$ and $N$ is large enough that $G$ must contain some member of $\mathcal{F}^{\prime}$.
	We will now show that $G$ has some subgraph $H_{k}$ on exactly $k$-vertices with $h_{k}(H_{k})\geq \alpha+\frac{\Delta}{2}$.

	Suppose that $G$ contains $F^{\prime}\in \mathcal{F}^{\prime}$.
	Note that $F^{\prime}$ (by construction) has $n\geq n_{0}$ vertices.
	Consider $G[F^{\prime}]$; we have that $h_{n}(G[F^{\prime}])\geq h_{n}(F^{\prime})\geq \alpha+\frac{\Delta}{2}$.
	Let $K$ be a random $k$-subset of the vertices of $G[F^{\prime}]$.
	Since $\mathbb{E}(h_{k}(G[K]))=h_{n}(G[F^{\prime}])\geq \alpha+\frac{\Delta}{2}$ it follows that there is some $k$-subset of $G[F^{\prime}]$
	satisfying $h_{k}(G[K])\geq \alpha+\frac{\Delta}{2}$.
	Then $G[K]$ is a $k$-vertex subgraph of $G$ with $h_{k}(G[K])\geq \alpha+\frac{\Delta}{2}$; we have found a subgraph with the desired properties.
	Hence $\alpha$ is a jump for $R$.
\end{proof}

\begin{prop}
	If $\alpha$ is a jump for $R$ and there exists an $R$-graph $F$ with $\lambda(F)=\alpha$ then $\alpha$ is a weak jump.
\end{prop}

\begin{proof}
	We need to show that $\alpha$ is not a strong jump.
	For $n$ larger than $|F|$, let $F_{n}$ be a blow-up of $F$ on $n$ vertices such that $h_{n}(F_{n})$ is as large as possible.
	The Lagrangian is constructed in such a way that $\displaystyle\lim_{n\to\infty}h_{n}(F_{n})=\lambda(F)$.
	Hence we have a sequence $\mathbf{F}=\{F_{n}\}$ with $h(\mathbf{F})=\lambda(F)=\alpha$.
	By construction, we also have that $\bar{h}(\mathbf{F})=\lambda(F)$.
	Hence $\alpha$ is not a strong jump.
\end{proof}

We observe the following easy proposition, which we give without proof.

\begin{prop}
	If $\alpha$ is a jump for $R$ and $\mathcal{F}$ is a finite family of $R$-graphs satisfying the conditions of theorem \ref{t:equiv},
	then every $\beta\in (\pi^{R}(\mathcal{F}), \min_{F\in\mathcal{F}}\lambda(F))$ is a strong jump for $R$.
\end{prop}


\section{Proof of  Theorem \ref{t:12}} 

We now give the proof of Theorem \ref{t:12}.  
The proof will make extensive use of Theorem \ref{t:equiv}.
Throughout this section we take $R=\{1,2\}$.

	\subsection*{Case: $\alpha\in [0,1)$}
		There is a unique integer $t\geq 2$ such that $\alpha\in [1-\frac{1}{t-1}, 1-\frac{1}{t})$.
		Let $\mathcal{F}=\{K_{1}^{\{1\}}, K_{t}^{\{2\}}\}$.
		If $G$ is an $R$-graph that is $\mathcal{F}$-free, then $G$ is also a $\{2\}$-graph.
		Hence \[\pi^{R}(\mathcal{F})=\pi^{\{2\}}(K_{t}^{\{2\}})=1-\frac{1}{t-1}\leq \alpha.\]
		We note that $\lambda(K_{1}^{\{1\}})=1$.
		We now compute $\lambda(K_{t}^{\{2\}})$.
		
		\begin{align*}
		\lambda(K_{t}^{\{2\}}) &= \max_{\vec{x}\in S_{t}} \lambda(K_{t}^{\{t\}},\vec{x}) \\
        		               &= \max_{\vec{x}\in S_{t}} \sum_{1\leq i<j\leq t} 2x_{i}x_{j} \\
							   &= \sum_{1\leq i<j\leq t}\frac{2}{t^{2}} \\
							   &= \binom{t}{2}\frac{2}{t^{2}} \\
							   &= \frac{t-1}{t} \\
							   &> \alpha.
		\end{align*}
		Note that the third line follows since every vertex is equivalent (see Lemma \ref{l:vertequiv}).
		Hence, by theorem \ref{t:equiv} we have that $\alpha$ is a jump.
		Furthermore, we have that for all $t\geq 2$ the value $1-\frac{1}{t-1}$ is a weak jump and every other value in $[0,1)$ is a strong jump.

	\subsection*{Case: $\alpha\in[1,\frac{9}{8})$}
		Let $F$ be a chain on two vertices, i.e. edges $\{1\}$ and $\{1,2\}$.
		Since $F$ is a chain (of length 2), $\pi^{R}(F)=1\leq \alpha$.
		We now compute $\lambda(F)$.
		\begin{align*}
			\lambda(F) &= \max_{\vec{x}\in S_{2}} \lambda(F, \vec{x}) \\
			           &= \max_{\vec{x}\in S_{2}} x_{1}+2x_{1}x_{2} \\
					   &= \max_{x_1\in [0,1]}x_{1}+2x_{1}(1-x_{1}) \\
					   &= \frac{9}{8}.
		\end{align*}
		Note that $1$ is not a strong jump by theorem \ref{turan} since $\pi^{R}(F)=1$.
		Also, note that $\frac{9}{8}$ is not a strong jump since $\lambda(F)=\frac{9}{8}$.

	\subsection*{Case: $\alpha\in [\frac{9}{8}, \frac{5}{4})$}
		Note that there is a unique $t\geq 3$ such that $\alpha\in [\frac{5}{4}-\frac{1}{4(t-1)}, \frac{5}{4}-\frac{1}{4t})$.
		Let $K^{\ast}_{t}$ denote the $\{1,2\}$-graph with vertex set $V=[t]$ and edge set $E=\{1\} \cup \binom{[t]}{2}$.
		First, we will show that $\pi (K^{\ast}_{t}, K_{2}^{\{1,2\}}) \leq \frac{5}{4} - \frac{1}{4(t-1)}$ (for $t\geq 3$).

		Let $G_{n}$ be any graph on $n$ vertices which forbids both $K^{\ast}_{t}$ and $K_{2}^{\{1,2\}}$.
		Partition the vertex set of $G_{n}$ into two sets $X$ and $\bar{X}$ where a vertex $v\in X$ if and only if $\{v\}\in E$.
		We will say that $|X|=xn$ and $|\bar{X}|=(1-x)n$.
		For each $v\in X$ define the set $N_{v}$ to be the set of vertices $u\in \bar{X}$ such that $\{v,u\}\in E$.
		Since $G_{n}$ is $K^{\ast}_{t}$-free, it follows that for each $v\in X$ the graph $G[N_{v}]$ is $K^{\{2\}}_{t-1}$-free.

		Note that $\pi(K_{t-1}^{\{2\}})=1-\frac{1}{t-2}$ so the number of edges in $G[N_{v}]$ is at most $\left(1-\frac{1}{t-2}\right)\binom{|N_{v}|}{2}+o(1)$.
		In other words, in $G[N_{v}]$ there are at least $\frac{1}{t-2}\binom{|N_{v}|}{2}-o(1)$ non-edges.
		Fix $v\in X$ such that $|N_{v}|$ is as large as possible, and say that $|N_{v}|=\alpha(1-x)n$.
		We then have that
		\begin{align*}
			h_{n}(G_{n}) &\leq \frac{xn}{\binom{n}{1}} + \frac{ \binom{(1-x)n}{2} + (xn)\alpha(1-x)n - \frac{1}{t-2}\binom{\alpha(1-x)n}{2}}{\binom{n}{2}} \\
            			 &= x + (1-x)^2 + 2\alpha x(1-x) - \frac{\alpha^2 (1-x)^2}{t-2} + o(1) \\
					     &= x + 2\alpha x(1-x) + (1-x)^2\left(1-\frac{\alpha^2}{t-2}\right) + o(1) \\
					     &\leq \max_{x\in [0,1], \alpha\in [0,1]} x+2\alpha x(1-x)+(1-x)^2\left(1-\frac{\alpha^{2}}{t-2}\right) + o(1) \\
					     &= \frac{5}{4} - \frac{1}{4(t-1)}+o(1).
		\end{align*}
		The last line of the inequality above is achieved when $x=\frac{t}{2(t-1)}$ and $\alpha=1$.
		With that observation, we have actually proven that $\pi(K_{t}^{\ast}, K_{2}^{\{1,2\}})=\frac{5}{4}-\frac{1}{4(t-1)}$.

		We now show that $\lambda(K_{2}^{\{1,2\}})=\frac{3}{2}>\alpha$.
		\begin{align*}
			\lambda(K_{2}^{\{1,2\}}) &= \max_{\vec{x}\in S_{2}} \lambda(K_{2}^{\{1,2\}}, \vec{x}) \\
			                         &= \max_{\vec{x}\in S_{2}} x_1+x_2+2x_1 x_2 \\
									 &= 1 + 2\left(\frac{1}{2}\right)\left(\frac{1}{2}\right) \\
									 &= \frac{3}{2}.
		\end{align*}
		Note again that the third line follows since the two vertices are equivalent.
		Now, we bound $\lambda(K_{t}^{\ast})$.

		Note that
			\[\lambda(K_{t}^{\ast}, \vec{x}) = x_1 + \sum_{1\leq i<j\leq t} 2x_{i}x_{j}.\]
		To get a lower bound on $\lambda(K_{t}^{\ast})$, choose $x_{1}=\frac{t+1}{2t}$ and $x_{i}=\frac{1}{2t}$ for $2\leq i\leq t$.
		It follows that 
		\begin{align*}
			\lambda(K_{t}^{\ast}) &\geq \frac{t+1}{2t} + 2\left(\frac{t+1}{2t}\right)\sum_{i=2}^{t}\frac{1}{2t} + \sum_{2\leq i<j\leq t}\left(\frac{1}{2t}\right)^2 \\
	    		                  &= \frac{t+1}{2t} + 2\left(\frac{t+1}{2t}\right)\left(\frac{t-1}{2t}\right)+2\binom{t-1}{2}\left(\frac{1}{2t}\right)^2 \\
								  &= \frac{5}{4}-\frac{1}{4t} \\
								  &> \alpha.
		\end{align*}
		Hence by theorem \ref{t:equiv} $\alpha$ is a jump for $R$.
		By the same arguments we've previously seen, the interior of the interval: $[\frac{5}{4}-\frac{1}{4(t-1)}, \frac{5}{4}-\frac{1}{4t})$ are strong jumps, and the endpoints are weak jumps.

	\subsection*{Case: $\alpha\in [\frac{5}{4}, \frac{3}{2})$}
		Note that $\pi^{R}(K_{2}^{\{1,2\}})=\frac{5}{4}\leq \alpha$.
		(This is a result from \cite{tnuh}.)
		Additionally, we already saw that $\lambda(K_{2}^{\{1,2\}})=\frac{3}{2}>\alpha$.
		Hence $\alpha$ is a jump.
		Again, $\frac{5}{4}$ and $\frac{3}{2}$ are weak jumps, and the interval $(\frac{5}{4},\frac{3}{2})$ is comprised of strong jumps.

		Finally, suppose that $\alpha\in [\frac{3}{2}, 2)$.
		Then there is a unique $t\geq 3$ such that $\alpha\in [2-\frac{1}{t-1}, 2-\frac{1}{t})$.
		Note that $\pi^{R}(K_{t}^{\{1,2\}})=2-\frac{1}{t-1}$; again, this is a result from \cite{tnuh}.
		Furthermore, every vertex is equivalent.
		Thus
		\begin{align*}
			\lambda(K_{t}^{\{1,2\}}) &= \max_{\vec{x}\in S_{t}} \lambda(K_{t}^{\{1,2\}}) \\
            			             &= \max_{\vec{x}\in S_{t}} \sum_{i=1}^{t} x_{i} + \sum_{1\leq i<j\leq t}2x_{i}x_{j} \\
									 &= 1 + 2\binom{t}{2}\left(\frac{1}{t}\right)^{2} \\
									 &= 1 + \frac{t-1}{t} \\
									 &= 2 - \frac{1}{t} \\
									 &> \alpha.
		\end{align*}
		Hence $\alpha$ is a jump.
		As before, $2-\frac{1}{t-1}$ and $2-\frac{1}{t}$ are weak jumps, and the interval $(2-\frac{1}{t-1}, 2-\frac{1}{t})$ is comprised of strong jumps.

\section{Hereditary Properties and Future Direction}

In this last section, we will reveal the relationship between the non-strong-jump values and hereditary properties.
Hereditary properties have been well-studied for graphs and $r$-uniform
hypergraphs \cite{BT95, BT97, DN, nikiforov1}.
This concept can be naturally extended to $R$-graphs.
A \textbf{property} of $R$-graphs is a family of $R$-graphs closed under isomorphism.
A property is called \textbf{hereditary} if it is closed under taking induced subgraphs.
A typical hereditary property can be obtained by forbidding a set of $R$-graphs as induced sub-hypergraphs.
Given a hereditary property $\mathcal{P}$ of $R$-graphs, let $\mathcal{P}_n$
be the set of $R$-graphs in $\cal P$ with $n$ vertices, and set
\[\pi_n(\mathcal{P})=\max_{G\in \mathcal{P}_n} h_n(G),\]
where $h_n(G)$ is the Lubell value of $G$.
We have the following proposition.
\begin{prop}\label{p7}
	For any hereditary property $\cal P$ of $R$-graphs, the limit 
	$\lim_{n\to \infty} \pi_n({\cal P})$ always exists.
\end{prop}
The limit, $\pi(\mathcal{P})$, is called the density of $\mathcal{P}$: 
\[\pi(\mathcal{P})=\lim_{n\to\infty}\pi_n(\mathcal{P}).\]
This proposition can be proved using the average argument, first shown in
Katona-Nemetz-Simonovit’s theorem \cite{KNS} for the existence of the
Tur\'an density of any $r$-uniform hypergraph. For non-uniform
hypergraphs, the proof
of Proposition \ref{p7} is actually
identical to the proof of existence of the Tur\'an density $\pi(\mathcal{H})$
(see Theorem 1 in \cite{tnuh}), and is omitted here. 
\begin{theorem}\label{t7}
	For any fixed set $R$ of finite positive integers, a value $\alpha \in [0, |R|]$ is not a strong jump
	for $R$ if and only if there exists a hereditary property $\mathcal{P}$ of $R$-graphs such that
	$\pi(\mathcal{P})=\alpha$.
\end{theorem}
\begin{proof}
By Corollary \ref{nsj}, $\alpha$ is not a strong jump for $R$ if and only if there exists an admissible
sequence of $R$-graphs $\mathbf{G}:=\{G_{n_i}\}_{i=1}^\infty$ satisfying $h(\mathbf{G})=\bar h(\mathbf{G})=\alpha$.

Now we show that it is a sufficient condition. 
Consider a hereditary property $\mathcal{P}$ with $\pi(\mathcal{P})=\alpha$.
Let $G_n\in\mathcal{P}_n$ be an $R$-graph achieving the maximum Lubell value, and $\mathbf{G}:=\{G_n\}$. 
By definition of $\pi(\mathcal{P})$, we have
\[h({\bf G})=\pi({\cal P})=\alpha.\]
Since $\mathcal{P}$ is hereditary, any induced subgraphs of $\mathbf{G}$ are still in $\mathcal{P}$.
Thus
\[\bar h(\mathbf{G})\leq \pi({\cal P})=\alpha.\]
Since $\bar h(\mathbf{G})\geq h(\mathbf{G})$, it forces $\bar h(\mathbf{G})=h(\mathbf{G})=\alpha$.
By Corollary \ref{nsj}, $\alpha$ is not a strong jump for $R$.

Now we show that it is also a necessary condition.
We define a property
\[\mathcal{P}:=\{H\colon H \text{ is an induced subgraph of } G_n \in \mathbf{G}\}.\]
It is clear that $\mathcal{P}$ is hereditary.
Since $\bar h(\mathbf{G})=\alpha$, we have $\pi(\mathcal{P})=\alpha$. 
\end{proof}

For $r$-uniform hypergraphs, Nikiforov \cite{nikiforov1} recently proved an
important result related to the $p$-spectrum and the density $\pi(\mathcal{P})$:
for any hereditary property $\mathcal{P}$ of $r$-graphs, and any $p>1$, $\lambda^{p}(\mathcal{P})=\pi(\mathcal{P})$.
Readers are referred to a recent survey paper \cite{nikiforov2} for
the exact definition of $p$-spectrum and $\lambda^{p}(\mathcal{P})$.
Our future work is to generalize the $p$-spectrum to non-uniform hypergraphs and study the spectral
characterization of the non-jump values. 
Another future task is to discover more non-jump values (even for the $r$-uniform hypergraphs). 
Some work has been done along this line. 
Frankl-Peng-R\"odl-Talbot \cite{FrPeRoTa07} proved that $\frac{5}{9}$ is a non-jump value for $3$ and made the
following conjecture:

\begin{center}
\begin{minipage}{0.9\linewidth}
\textbf{Conjecture (see \cite{FrPeRoTa07}):} 
\textit{For any $l\geq 3$, $s \geq 1$, and $l\geq s+1$, the value 
 $1-\frac{3}{l}+\frac{3s+2}{l^2}$ is a non-jump value for $r=3$.}
\end{minipage}
\end{center}
The value $\frac{5}{9}$ corresponds to $l=3$ and $s=1$.
They showed the conjecture holds for all $l\geq 9s+6$.
We \cite{Lag-UNI} made some progress on this conjecture and found several new non-jump values for $r=3$.

\end{document}